\documentclass[a4paper,twoside,12pt]{article}
\usepackage{amssymb,amsmath,amsthm,latexsym}
\usepackage{amsfonts}
\usepackage{graphicx,caption,subcaption}
\usepackage[pdftex,bookmarks,colorlinks=false]{hyperref}

\usepackage[hmargin=1.25in,vmargin=1.25in]{geometry}

\usepackage[auth-sc-lg,affil-sl]{authblk}
\newtheorem{thm}{Theorem}[section]

\newtheorem{defn}[thm]{Definition}

\newtheorem{lem}[thm]{Lemma}

\newtheorem{prob}[thm]{Problem}

\def\ssgm{\sigma^{\ast}}
\def\hsgm{\sigma^{\#}}
\def\ni{\noindent}

\title{\textbf{\sc Weak Set-Labeling Number of Certain Integer Additive Set-Labeled Graphs}}

\author{N.K. Sudev}
\affil{\small Department of Mathematics\\ Vidya Academy of Science \& Technology \\ Thalakkottukara, Thrissur - 680501, Kerala, India.\\ E-mail: sudevnk@gmail.com}

\author{K. A. Germina}
\affil{\small PG \& Research Department of Mathematics\\ Mary Matha Arts \& Science College\\Mananthavady, Wayanad-670645, Kerala, India.\\ E-mail: srgerminaka@gmail.com}

\author{K. P. Chithra}
\affil{\small Naduvath Mana, Nandikkara\\ Thrissurd-680301, Kerala, India.\\ E-mail: chithrasudev@gmail.com}

\date{}

\begin{document}
\maketitle

\begin{abstract} 
Let $\mathbb{N}_0$ be the set of all non-negative integers, let $X\subset \mathbb{N}_0$ and $\mathcal{P}(X)$ be the the power set of $X$. An integer additive set-labeling (IASL) of a graph $G$ is an injective function $f:V(G)\to \mathcal{P}(\mathbb{N}_0)$ such that the induced function $f^+:E(G) \to \mathcal{P}(\mathbb{N}_0)$ is defined by $f^+ (uv) = f(u)+ f(v)$, where $f(u)+f(v)$ is the sum set of $f(u)$ and $f(v)$. An IASL $f$ is said to be an integer additive set-indexer (IASI) of a graph $G$ if the induced edge function $f^+$ is also injective. An integer additive set-labeling $f$ is said to be a weak integer additive set-labeling (WIASL) if $|f^+(uv)|=\max(|f(u)|,|f(v)|)~\forall ~ uv\in E(G)$. The minimum cardinality of the ground set $X$ required for a given graph $G$ to admit an IASL is called the set-labeling number of the graph.  In this paper, we introduce the notion of the weak set-labeling number of a graph $G$ as the minimum cardinality of $X$ so that $G$ admits a WIASL with respect to the ground set $X$ and discuss the weak set-labeling number of certain graphs.
\end{abstract}

\ni {\bf Keywords:} Integer additive set-labeled graphs; weak integer additive set-labeled graphs; weak set-labeling number of a graph.

\ni \textbf{AMS Subject Classification: 05C78}

\section{Introduction}

For all  terms and definitions, not defined specifically in this paper, we refer to \cite{BM}, \cite{FH} and \cite{DBW} and for different graph classes, we further refer to \cite{BLS} and \cite{JAG}. Unless mentioned otherwise, all graphs considered here are simple, finite and have no isolated vertices.

\vspace{0.2cm} 

The sum set of two sets $A$ and $B$, denoted $A+B$, is the set defined by $A + B = \{a+b: a \in A, b \in B\}$. If either $A$ or $B$ is countably infinite, then their sum set will also be countably infinite. Hence, all sets we consider in this study are finite sets. The cardinality of a set $A$ is denoted by $|A|$. The power set of  a set $A$ is denoted by $\mathcal{P}(A)$.

\vspace{0.2cm}

Using the concepts of sumsets, the notion of an integer additive set-labeling of a graph $G$ is introduced as follows.

\vspace{0.2cm} 

Let $\mathbb{N}_0$ denote the set of all non-negative integers and $X$ be a subset of $\mathbb{N}_0$. An {\em integer additive set-labeling} (IASL, in short) of a graph $G$ is defined as an injective function $f:V(G)\to \mathcal{P}(X)$ such that the induced function $f^+:E(G) \to \mathcal{P}(X)$ is defined by $f^+ (uv) = f(u)+ f(v)$, where $f(u)+ f(v)$ is the sumset of the set-labels $f(u)$ and $f(v)$. A graph which admits an IASL is called an {\em integer additive set-labeled graph} (IASL-graph).

\vspace{0.2cm} 

An {\em integer additive set-indexer} of a graph $G$ is an integer additive set-labeling $f:V(G)\to \mathcal{P}(X)$ such that the induced edge function $f^+:E(G)\to \mathcal{P}(X)$ is also injective.

\vspace{0.2cm} 

Figure \ref{fig:G-IASIG1} depicts an integer additive set-labeling defined on a given graph $G$.

\begin{figure}[h!]
\centering
\includegraphics[width=0.55\linewidth]{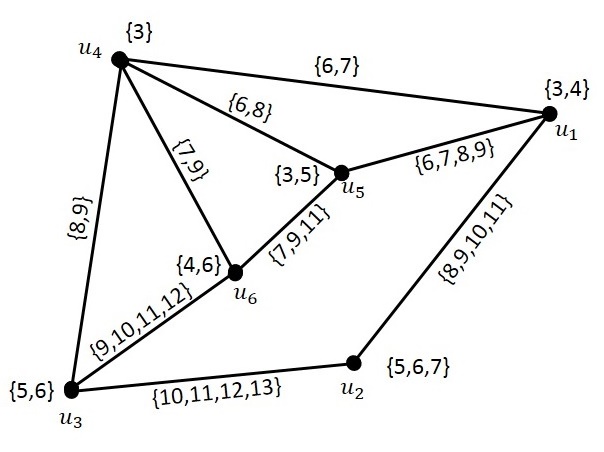}
\caption{An illustration to an IASL-graph.}
\label{fig:G-IASIG1}
\end{figure}

\vspace{0.2cm} 

The {\em set-labeling number} of a given graph $G$ is the minimum cardinality of the ground set $X$ so that the function $f:V(G)\to \mathcal{P}(X)$ is a WIASL of $G$. The set-labeling number of a graph $G$ is denoted by $\sigma(G)$. 

\vspace{0.2cm}

\begin{defn}{\rm
\cite{GS3} A \textit{weak integer additive set-labeling} $f$ of a graph $G$ is an IASL such that $|f^+(uv)|= \max(|f(u)|,|f(v)|)$ for all $u,v\in V(G)$. A weak IASI $f$ is said to be {\em weakly uniform IASI} if $|f^+(uv)|=k$, for all $u,v\in V(G)$ and for some positive integer $k$.  A graph which admits a weak IASI may be called a {\em weak integer additive set-labeled graph} (WIASL-graph).}
\end{defn}

A WIASL $f$ of a given graph $G$ is said to be a {\em weakly $k$-uniform IASL} ($k$-uniform WIASL) of $G$ if the set-labels of all edges of $G$ have the same cardinality $k$, where $k$ is a positive integer. If $G$ admits a WIASL, then it can be noted that the vertex set of $G$ can be partitioned into two sets such that the first set, say $V_1$, consists of all those vertices of $G$ having singleton set-labels and the other set, say $V_2$ consists of all those vertices having non-singleton set-labels. As a result, we have the following theorem.

\begin{thm}
{\rm \cite{GS1}} A graph $G$ admits a $k$-uniform WIASL if and only if $G$ is bipartite or $k=1$.
\end{thm}

\vspace{0.2cm} 

The following result is a necessary and sufficient condition for a given graph to admit a weak integer additive set-labeling. 

\begin{lem}
{\rm \cite{GS3}} A graph $G$ admits a weak integer additive set-indexer if and only if every edge of $G$ has at least one mono-indexed end vertex.
\end{lem}

In view of the above lemma, no two of its adjacent vertices can have non-singleton set-labels. Then, some of the adjacent vertices of $G$ can have singleton sets as their end vertices. In these cases, the set-label of the corresponding edges are also singleton sets. As a result, we have the following theorem.

\begin{thm}
{\rm \cite{GS3}} A graph $G$ admits a WIASL if and only if $G$ is bipartite or it has some edges having singleton set-label.
\end{thm}

In view of the above results, we can re-define a WIASL $f$ of a given graph $G$ as an IASL with respect to which the cardinality of the set-label of every edge of $G$ is equal to the cardinality of the set-label of at least one of its end vertices.

\vspace{0.2cm}

The following result is another important observation we have on WIASL-graphs.

\begin{thm}\label{T-INCN}
{\rm \cite{GS8}} Let $G$ be a WIASL-graph. Then, the minimum number of vertices of $G$ having singleton set-labels is equal to the vertex covering number $\alpha$ of $G$ and the maximum number of vertices of $G$ having non-singleton set-labels is equal to the independence number $\beta$ of $G$.
\end{thm}

\vspace{0.2cm} 

Figure \ref{fig:G-WIASL1} illustrates a WIASL defined on a given graph $G$. 

\begin{figure}[h!]
\centering
\includegraphics[width=0.5\linewidth]{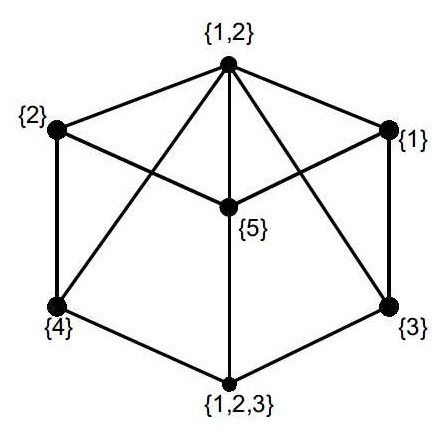}
\caption{An example to a WIASL-Graph}
\label{fig:G-WIASL1}
\end{figure}

\section{Weak Set-Labeling Number of Graphs}

Analogous to the terminology of set-labeling number of graphs which admit integer additive set-labelings, we introduce the notion of weak set-labeling number of a  given graph $G$ as follows.

\begin{defn}{\rm 
Let a function $f:V(G)\to \mathcal{P}(X)$ be an IASL of a given graph $G$, where $X$ is a non-empty finite ground set of non-negative integers. Then, the {\em weak set-labeling number} of a graph $G$ is the minimum cardinality of the ground set $X$, such that $f$ is a WIASL of $G$. The weak set-labeling number of a graph $G$ is denoted by $\ssgm(G)$.}
\end{defn}

No IASL-graphs we consider in the following discussion are $1$-uniform, unless specified otherwise.

In the following discussion, we determine the weak set-labeling number of different standard graphs. Let us begin with a path graph $P_n$ on $n$ vertices.

\begin{thm}\label{T-WSN-Pn}
The weak set-labeling number of a path $P_n$ is $2+\lfloor \frac{n}{2}\rfloor$.
\end{thm}
\begin{proof} 
Let $\{v_i,v_2,v_3,\ldots, v_n\}$ be the vertex set of the graph $P_n$. Since we need a set-labeling in such a way that no two adjacent vertices in $P_n$ can have non-singleton set-labels. Hence, we need to label the vertices alternately by singleton and non-singleton sets of non-negative integers. In this context, we need to consider the following cases.

\vspace{0.1cm}

{\bf Case-1:} Let $n$ be even. Then, $n=2r$ for some positive integer $r$. Label the vertices of $P_n$ by singleton sets and $2$-element sets of non-negative integers as follows. Let $f(v_2)=\{1\}, f(v_4)=\{2\}, f(v_6)=\{3\},\ldots, f(v_{2r})=\{r\}$. Now, label odd vertices in reverse order as follows. Let $f(v_{2r-1})=\{1,2\}, f(v_{2r-3})=\{2,3\}, f(v_{2r-5})=\{3,4\},\ldots, v_3=\{r-1,r\}, f(v_1)=\{r,1\}$.  Then, the set-labels of the edges of $G$ are  $v_nv_{n-1}=\{r+1,r+2\}$  $v_{n-1}v_{n-2}=\{1+r-1,2+r-1\}=\{r+1,r+2\}$ and so on. The set-labels of all edges other than the edge $v_1v_2$ are either \{r+1,r+2\}$ or \{r,r+1\}$ and the set-label of $v_1v_2$ is $\{2,r+1\}$. That is, $f^+(E)=\{\{2,r+1\}, \{r, r+1\}, \{r+1,r+2\}\}$. Now, choose $X=f(V)\cup f^+(E)=\{1,2,3,\ldots,r+1,r+2\}$. Clearly, $f:V(G)\to \mathcal{P}(X)$ is a WIASL defined on $P_n$ and hence, $\ssgm (P_n)=2+\frac{n}{2}$. 

\vspace{0.1cm}

{\bf Case-2:} Let $n$ be odd. Then, $n=2r+1$, for some positive integer $r$. Label the vertices of $P_n$ as follows. Let $f(v_2)=\{1\}, f(v_4)=\{2\}, f(v_6)=\{3\},\ldots, f(v_{2r})=\{r\}$. Now, label the remaining vertices in the reverse order as follows. Let $f(v_{2r+1})=\{1,2\}, f(v_{2r-1})=\{2,3\}, f(v_{2r-3})=\{3,4\},\ldots, v_5=\{r-1,r\}, f(v_3)=\{r,1\}$ and $f(v_1)=\{1,2,3\}$. Therefore, as in the previous case, $f^+(E)=\{\{2,3,4\},\{r,r+1\},\{r+1,r+2\}\}$. Then, $X=f(V)\cup f^+(E)=\{1,2,3,\ldots, r+1,r+2\}$. Therefore, the weak set-labeling number of $P_n$ is $r+2= 2+\frac{n-1}{2}$. 

\vspace{0.1cm}

Combining the above two cases, we have $\ssgm(P_n)=2+\lfloor \frac{n}{2}\rfloor$.
\end{proof}

\vspace{0.1cm}

Next, we proceed to determine the weak set-labeling number of cycle graphs. 

\begin{thm}\label{T-WSN-Cn}
The weak set-labeling number of a cycle $C_n$ is $2+\lfloor \frac{n}{2}\rfloor$.
\end{thm}
\begin{proof} 
Let $C_n:v_iv_2v_3\ldots v_nv_1$ be a cycle on $n$ vertices. We label the vertices alternately by singleton and non-singleton sets of non-negative integers. Here, we have the following cases.

\vspace{0.1cm}

{\bf Case-1:} Let $n$ be even. Then, $n=2r$ for some positive integer $r$. Label the vertices of $C_n$ by singleton sets and $2$-element sets of non-negative integers as follows. Let $f(v_{2i})=\{i\}$ for all $1\le i\le \frac{n}{2}$. Label the odd vertices such that $f(v_{2r-1})=\{1,2\}, f(v_{2r-3})=\{2,3\}, f(v_{2r-5})=\{3,4\},\ldots, v_3=\{r-1,r\}, f(v_1)=\{r,1\}$. Then, all edges of $C_n$ except  $v_1v_2$ and $v_2v_3$ have the set-labels either $\{r,r+1\}$ or $\{r+1,r+2\}$. Also, $f(v_1v_2)=\{2,r+1\}$ and $f(v_2v_3)=\{3,r+2\}$. Therefore, Choose $X= \{1,2,3,\ldots,r+1,r+2\}$. Hence, in this case $\ssgm(C_n)= r+2= 2+\frac{n}{2}$.

\vspace{0.1cm}

{\bf Case-2:}  Let $n$ be odd. Then, $n=2r+1$, for some positive integer $r$. For an WIASL-graph $C_n$, there exist $r+1$ singleton set labels and $r$ non-singleton sets. Label the vertices of $C_n$ as follows. Now, let $f(v_1)=\{1\}, f(v_3)=\{2\}, f(v_5)=\{3\},\ldots, f(v_{2r+1})=\{r+1\}$ and let $f(v_2)=\{r,r-1\},f(v_4)=\{r-1,r-2\},f(v_6)=\{r-2,r-3\}, \ldots f(v_{2r})=\{1,2\}$. Then, all edges other than the edge $v_nv_1$ have the set-label either $\{r,r+1\}$ or $\{r+1,r+2\}$ and the edge $v_nv_1$ has the set-label $\{r+2\}$. Therefore, we can choose $X$ as the set $f(V)\cup f^+(E)=\{1,2,3,\ldots, r+1, r+2\}$ and hence $\ssgm (G)=r+2= 2+\frac{n-1}{2}$.

\vspace{0.1cm}

Combining the above two cases, we have $\ssgm(C_n)=2+\lfloor \frac{n}{2}\rfloor$.
\end{proof}

\vspace{0.2cm}

The weak set-labeling number of a complete graph is determined in the following theorem.

\begin{thm}
The weak set-labeling number of a complete graph $K_n$ is $2n-3$.
\end{thm}
\begin{proof} 
Let $V(K_n)=\{v_1,v_2,v_3,\ldots,v_n\}$. Since all vertex of $K_n$ are mutually adjacent, at most one vertex can have a non-singleton set-label. Let $f$ be an IASL defined on $K_n$ which assigns the set-labels to the vertices of $K_n$ as follows. Let $f(v_i)=\{i\}$ for $1\le i \le n-1$ and let $f(v_n)=\{1,2\}$. Then, $f^+(E(K_n))=\{3,4,5,\ldots,2n-3\}$. That is, the ground set with minimum cardinality is the set $f(V)\cup f^+(E)=\{1,2,3,\ldots,2n-3\}$. Hence, $\ssgm(G)=2n-3$. \end{proof}

\vspace{0.2cm}

Figure \ref{fig:G-KN-WSN} depicts a complete graph on $6$ vertices and weak set-labeling number $9$.

\begin{figure}[h!]
\centering
\includegraphics[width=0.5\linewidth]{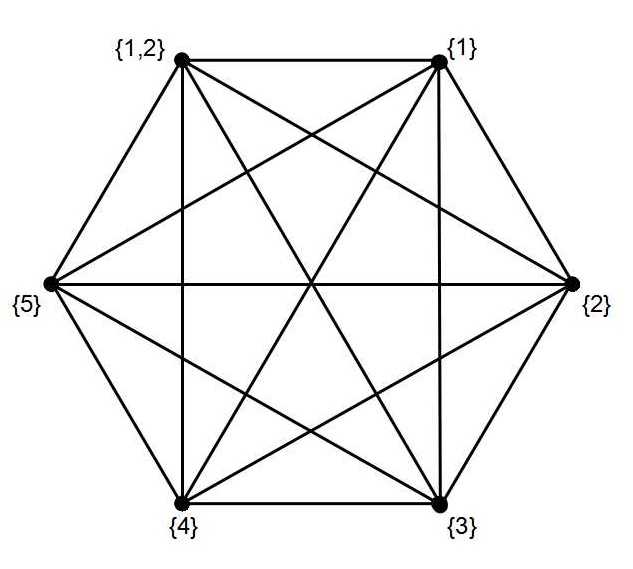}
\caption{A complete graph with weak set-labeling number $9$.}
\label{fig:G-KN-WSN}
\end{figure}

\section{Weak Set-Labeling Number of Certain Graph Classes}

Let us now proceed to discuss the weak set-labeling number of certain graphs that are generated from cycles. The first graph we consider among these types of graphs is a {\em wheel graph} $W_{n+1}$ which is a graph obtained by drawing edges from all vertices of a cycle to an external vertex (see \cite{JAG}). That is, $W_{n+1}=C_n+K_1$. The following theorem establishes the weak set-labeling number of a wheel graph.

\begin{thm}
The weak set-labeling number of a wheel graph $W_{n+1}=C_n+K_1$ is $3+\lfloor \frac{n}{2}\rfloor$.
\end{thm}
\begin{proof} 
Let $G=C_n+K_1$, where $C_n:v_1v_2v_3\ldots v_nv_1$ and $K_1=\{v\}$. Label the vertices of $C_n$ in $G$ as explained in Theorem \ref{T-WSN-Cn}.  What remains is to label the vertex $v$. Since all sets of the form $\{i,i+1\}$ have already been used for labeling the vertices of $C_n$, label $v$ by the set $\{1,3\}$. It can be observed that the only element in a set-label of the edge $vv_i$ of $G$, which is not in any set-label of the elements of $C_n$ is $r+3$, where $r=\lfloor \frac{n}{2}\rfloor$. Hence, $f(V(G))\cup f^+(E(G))=\{1,2,3,\ldots, r+3\}$. Therefore, we have the weak set-labeling number of the wheel graph $G$ is $3+\lfloor \frac{n}{2}\rfloor$. 
\end{proof}

\vspace{0.25cm}

Figure \ref{fig:G-WG-WSN} illustrates the wheel graph $W_7$ with weak set-labeling number $6$.

\begin{figure}[h!]
\centering
\includegraphics[width=0.5\linewidth]{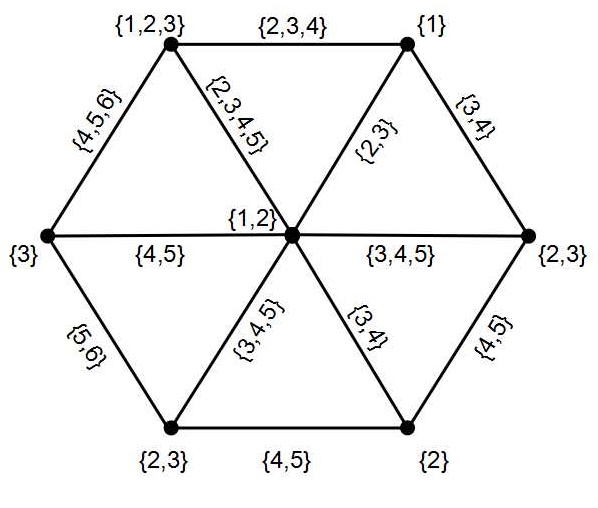}
\caption{A wheel graph with weak set-labeling number $6$.}
\label{fig:G-WG-WSN}
\end{figure}

\vspace{0.2cm}

A {\em helm graph}, (see \cite{JAG}) denoted by $H_n$, is a graph obtained by attaching a pendant edge to each vertex of the outer cycle of a wheel graph $W_{n+1}$. Then, the helm graph $H_n$ has $2n+1$ vertices and $3n$ edges. That is, $H_n=W_{n+1}\odot K_1$, where $\odot$ is the corona of two graphs. The following theorem establishes the weak set-labeling number of the helm graph $H_n$. 

\begin{thm}
The weak set-labeling number of a helm graph $H_n$ is $n+3$.
\end{thm}
\begin{proof} 
Let $G$ be a helm graph on $2n+1$ vertices. Then, $G=(C_n+K_1)\odot K_1$. Let $\{v\}$ be the central vertex, $\{v_1,v_2,v_3,\ldots, v_n\}$ be the vertices in the cycle $C_n$ of $G$ and let $\{u_1,u_2,u_3,\ldots, u_n\}$ be the pendant vertices of $G$ such that the vertex $u_i$ is adjacent to the vertex $v_i$, $1\le i\le n$. Since a vertex cover of $H_n$ contains $n$ elements, by Theorem \ref{T-INCN}, with respect to any WIASL defined on $G$, there must be at least $n$ vertices in $G$ must have singleton set-labels. Now, define an IASL $f$ on $G$ as follows. Label the vertices of $C_n$ by $f(v_i)=\{i\}$, where $1\le i\le n$. Then, label the pendant vertices of $G$ such that $f(u_n)=\{1,2\}, f(u_{n-1})=\{2,3\}, f(u_{n-2})=\{3,4\},\ldots, f(u_2)=\{n-1,n\}, f(u_1)=\{1,n\}$. The only vertex that remains to be labeled is the central vertex $v$. Since all subsets of $X$ of the form $\{i,i+1\}$ have been used for labeling other vertices, choose the set $\{1,3\}$ to label the vertex $v$. Then, $f(V(G))\cup f^+(E(G))=\{1,2,3,\ldots,n+3\}$. Since the minimal ground set $X$ is $f(V(G))\cup f^+(E(G))$, then $\ssgm(G) = n+3$. 
\end{proof}

\vspace{0.2cm}

Figure \ref{fig:G-WG-WSN} illustrates the helm graph on $13$ vertices with weak set-labeling number $7$.

\begin{figure}[h!]
\centering
\includegraphics[width=0.55\linewidth]{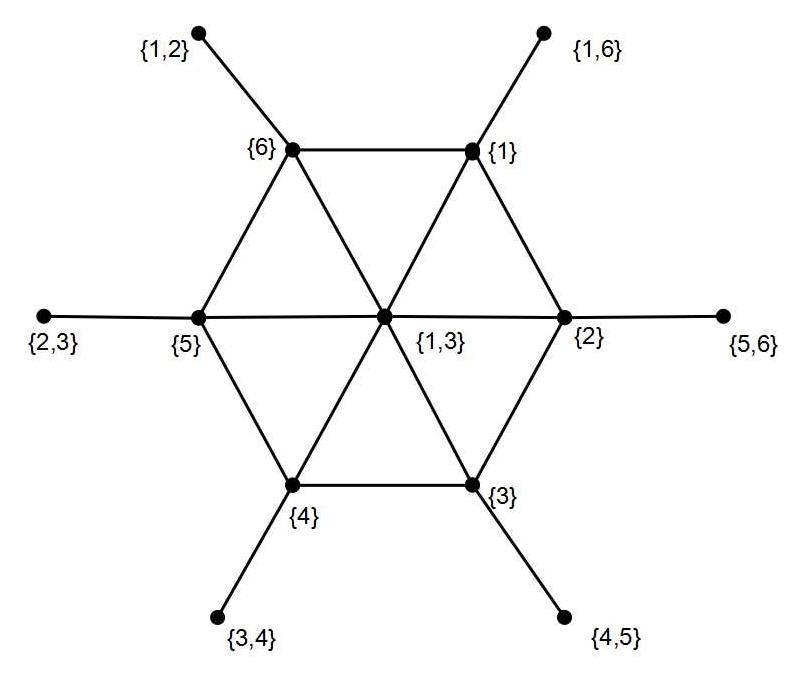}
\caption{A helm graph with weak set-labeling number $9$.}
\label{fig:G-HG-WSN}
\end{figure}

A {\em friendship graph} $F_n$ is a graph obtained by identifying one end vertex of $n$ triangles. A friendship graph has $2n+1$ vertices and $3n$ edges. The following theorem determines the weak set-labeling number of a friendship graph.

\begin{thm}
The weak set-labeling number of a friendship graph $F_n$ is $n+3$.
\end{thm}
\begin{proof} 
Let $v$ be the central vertex of the friendship graph $F_n$. Divide the remaining vertices of $F_n$ in to two sets $U=\{u_1u_2,u_3,\ldots, u_n\}$ and $W=\{w_1,w_2,w_3,\ldots,w_n\}$ such that the graph induced by the vertices $v, u_i,w_i$ is a triangle in $F_n$, where $1\le i\le n$. Let $f$ be an IASL defined on $F_n$ with respect to which, the set-labeling of the vertices of $F_n$ is done in the following way. Let $f(v)=\{1\}$ and $f(u_i)=\{i+1\}$, where $1\le i\le n$. Now, label the set $W$ in such a way that $f(w_n)=\{1,2\}, f(w_{n-1})=\{2,3\}, \ldots, f(w_2)=\{n-1,n\}, f(w_1)=\{n,n+1\}$.
Then, $f^+(E)=\{2,3,4,\ldots, n+3\}$. Hence, the minimal ground set is $f(V)\cup f^+(E)=\{1,2,3,\ldots,n+3\}$ and $\ssgm(G)=n+3$.
\end{proof}

\vspace{0.2cm}

Figure \ref{fig:G-FG-WSN} illustrates a friendship graph $F_4$ with weak set-labeling number $7$.

\begin{figure}[h!]
\centering
\includegraphics[width=0.5\linewidth]{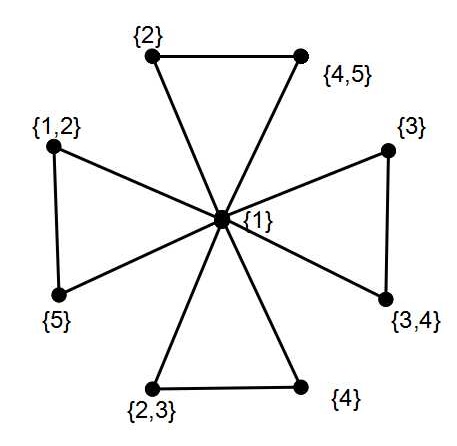}
\caption{A friendship graph with weak set-labeling number $7$.}
\label{fig:G-FG-WSN}
\end{figure}

\vspace{0.2cm}

A {\em sunlet graph} $SL_n$ is a graph obtained by attaching a pendant edge to all vertices of a cycle $C_n$ (see \cite{WDW}). The sunlet graph $SL_n$ has $2n$ vertices and $2n$ edges. The following theorem establishes the weak set-labeling number of a sunlet graph.

\begin{thm}
The weak set-labeling number of a sunlet graph $SL_n=C_n\odot K_1$ is $n+2$. 
\end{thm}
\begin{proof} 
We have $G=C_n\odot K_1$. Let $\{v_1,v_2,v_3,\ldots, v_n\}$ be the vertices of $C_n$ and $\{u_1,u_2,u_3\ldots,u_n\}$ be the pendant vertices of $G$. Let $f$ be be an IASL defined on $SL_n$ which assigns set-labels to the vertices of $SL_n$ as follows. Let $f(u_n)=\{1,2\}, f(u_{n-1})=\{2,3\}, f(u_{n-2})=\{3,4\},\ldots, f(u_2)=\{n-1,n\}, f(v_1)=\{1,n\}$. Then, as explained in previous theorems, the minimal ground set $X=f(V(G))\cup f^+(E(G))=\{1,2,3,\ldots,n+2\}$. Hence, for the sunlet graph $G=C_n\odot K_1$, $\ssgm(G)=n+2$. \end{proof}

\vspace{0.2cm}

Another similar graph, we consider here, is a sun graph which is defined as follows. A {\em sun graph} $S_n$ is a graph obtained by replacing every edge of a cycle $C_n$ by a triangle $C_3$ (see \cite{BLS}). A sun graph also has $2n$ vertices. The same set-labeling, as explained in the previous theorem, can be applied to the vertices of a sun graph $S_n$ and hence we have the following theorem.

\begin{thm}\label{T-WSN-SG}
The weak set-labeling number of a sun graph $S_n$ is $n+3$.
\end{thm}
\begin{proof} 
Let $\{v_1, v_2,v_3,\ldots, v_n\}$ be the vertex set of the cycle $C_n$ and let $\{u_1,u_2,u_3,\ldots,u_n\}$ be the independent vertices of $S_n$, such that the vertex $u_i$ is adjacent to the vertices $v_i$ and $v_{i+1}$, in the sense that $v_{n+1}=v_1$. As explained in the previous theorem, label the vertices of $C_n$ as $f(v_i)=\{i\}$ and label the independent vertices such that $f(u_n)=\{1,2\}, f(u_{n-1})=\{2,3\}, f(u_{n-2})=\{3,4\},\ldots, f(u_2)=\{n-1,n\}, f(v_1)=\{1,n\}$. Then, $f^+(E(S_n))=\{2,3,4,\ldots, n+3\}$. Hence, the minimal ground set $X=f(V)\cup f^+(E)=\{1,2,3,\ldots, n+3\}$. Therefore, the weak set-labeling number of the graph $S_n$ is $n+3$. 
\end{proof}

\vspace{0.2cm}

Figure \ref{fig:G-SN-WSN} illustrates Theorem \ref{T-WSN-SG}. 

\begin{figure}[h!]
\centering
\includegraphics[width=0.55\linewidth]{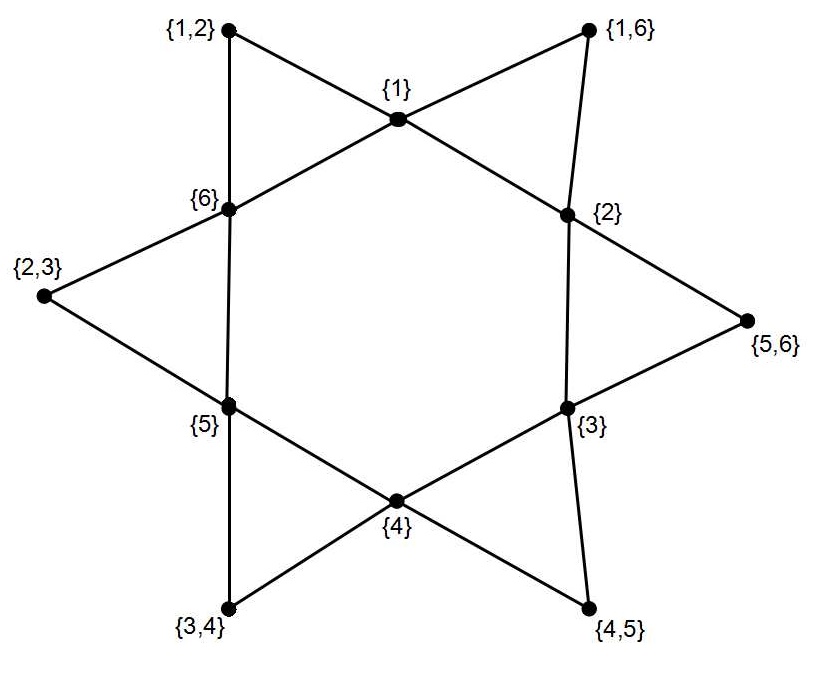}
\caption{A sun graph with weak set-labeling number $6$}
\label{fig:G-SN-WSN}
\end{figure}

The next graph we study for weak set-labeling number is a complete sun graph which is defined as follows. A {\em sun} (or {\em trampoline}) (see \cite{BLS}) is a $G$ on $n$ vertices for some $n>3$ whose vertex set can be partitioned into two sets, $W = \{w_1,w_2,w_3,\ldots, w_n\}$, $U = \{u_1,u_2,u_3,\ldots,u_n\}$, such that $W$ is independent and for each $i$ and $j$, the vertex $w_i$ is adjacent to the vertices $v_i$ and $v_{i+1}$, in the sense that $v_{n+1}=v_1$. A {\em complete sun} is a sun $G$ in which the induced subgraph $\langle U \rangle$ is a complete graph. The following theorem determines the weak set-labeling number of a complete sun graph. A sun (or complete sun) has $2n$ vertices.

\begin{thm}
The weak set-labeling number of a complete sun graph on $2n$ vertices is $n+3$.
\end{thm}
\begin{proof} 
Let $G$ be a complete graph on $2n$ vertices whose vertex sets are partitioned in to two sets $U$ and $W$, where the induced subgraph $\langle U \rangle$ of $G$ is a complete graph and $W$ is an independent set. As usual, label the vertices of $C_n$ as $f(u_i)=\{i\}$ and label the vertices in $W$ in such a way that $f(w_n)=\{1,2\}, f(w_{n-1})=\{2,3\}, f(w_{n-2})=\{3,4\},\ldots, f(w_2)=\{n-1,n\}, f(w_1)=\{1,n\}$. Then, $f^+(E)=\{2,3,4,\ldots, n+3\}$ and $f(V)\cup f^+(E)=\{1,2,3,\ldots, n+3\}$. Hence, the minimum required cardinality for the ground set $X$ is $\ssgm(G)=|f(V)\cup f^+(E)|=n+3$. 
\end{proof}

\vspace{0.2cm}

Figure \ref{fig:G-KS-WSN} illustrates a complete sun graph with a weak set-labeling number $7$.

\begin{figure}[h!]
\centering
\includegraphics[width=0.55\linewidth]{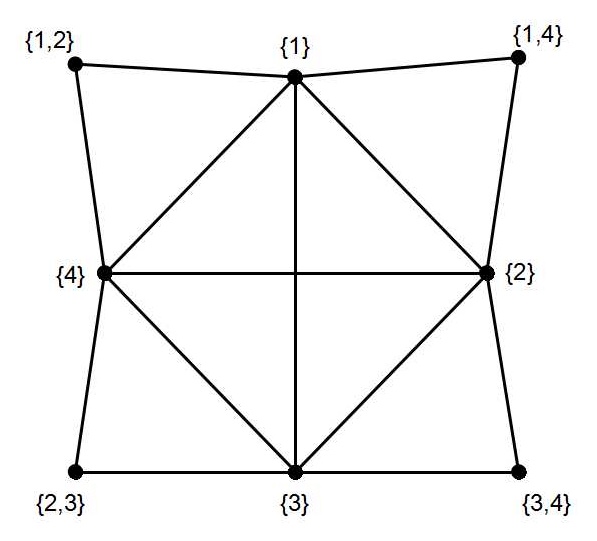}
\caption{A complete sun graph with weak set-labeling number $7$}
\label{fig:G-KS-WSN}
\end{figure}

\section{Scope for further studies}

In this paper, we have discussed the weak set-labeling number and weak set-indexing number of certain classes of graphs. Certain Problems in this area are still open. Some of the open problems we have identified in this area are the following.

\begin{prob}{\rm 
Determine the weak set-labeling number of a bipartite graphs, both regular and biregular.}
\end{prob}

\begin{prob}{\rm 
Determine the weak set-labeling number of certain graphs like gear graphs, lobsters, double wheel graphs etc.}
\end{prob}

\begin{prob}{\rm 
Determine the weak set-labeling number of certain graphs like double helm graphs, web graphs, windmill graphs etc.}
\end{prob}

\begin{prob}{\rm 
Determine the weak set-labeling number of split graphs, complete split graphs, bisplit graphs etc.}
\end{prob}

\begin{prob}{\rm 
Determine the weak set-labeling number of arbitrary sun graphs, rising sun graphs and partial sun graphs etc.}
\end{prob}

\begin{prob}{\rm 
Determine the weak set-labeling number of graphs, which admit a $1$-uniform IASL.}
\end{prob}

Analogous to the weak set-labeling number of graphs, we can 
define the weak set-indexing number of a graph $G$ as follows.

\begin{defn}
The minimum cardinality of the ground set $X$, so that the function $f:V(G)\to \mathcal{P}(X)$ is a WIASI of $G$, is called the weak set-indexing number of $G$ and is denoted by $\hsgm(G)$.
\end{defn}

Determining the weak set-indexing number of different graph classes is also an open problem.

There are several other types of standard graphs and named graphs whose weak set-labeling numbers and weak-set-indexing numbers can be calculated. Determining these parameters of different graph operations and graph parameters are also promising. Similar parameters corresponding to other types of integer additive set-labelings such as strong integer additive set-labelings, arithmetic integer additive set-labelings, exquisite integer additive set-labelings etc. are also worthy for future studies. All these facts show that there are wide scope for further studies in this direction.

\end{document}